\documentclass[11pt]{article}

\usepackage{mathrsfs}
\usepackage{graphicx,  amsmath,amsfonts,titlesec,cite,color,
                                                                                   }
 \allowdisplaybreaks[4]   

\def\BMO{{\mathrm{BMO}}}
\def\loc{{\mathrm{loc}}}

\numberwithin{equation}{section} 

\usepackage{indentfirst} 
\usepackage[amsmath,thmmarks]{ntheorem} 

 \theoremstyle{empty}
 \newtheorem{refproof}{Proof}  
\usepackage{extarrows}  
 \usepackage[sort&compress,numbers]{natbib}  
\usepackage{hyperref}
 \hypersetup{
    colorlinks=true,      
    citecolor=magenta,    
}
 \usepackage[capitalise]{cleveref}
 \usepackage[shortlabels]{enumitem}  
\def\corrauth{\footnote{Corresponding author.
 }
 \stepcounter{footnote}
 }
\def\affilnum#1{${}^{#1}$}
\def\affil#1{${}^{#1}$}
\def\comma{$^{\textrm{,}}$}

\theoremstyle{definition} 
\theoremheaderfont{\normalfont\rmfamily\bf}
\theorembodyfont{\normalfont\rm } \theoremindent0em
\newtheorem{theorem}{\indent
                  Theorem}[section]
    \newtheorem{lemma}{\indent  Lemma} [section]
\newtheorem{corollary}{\indent Corollary} [section]

    \newtheorem{definition}{\indent  Definition} [section]

\theoremstyle{nonumberplain} 
\theoremheaderfont{\bf\rmfamily} \theorembodyfont{\normalfont \rm }
\theoremsymbol{\hfill $\Box$} 
\newtheorem{proof}{\indent Proof}

\makeatletter
\renewcommand*{\author}[2][?]{
     \gdef\shortauthor{?}
     \gdef\@author{#2}
   \ifthenelse{\equal{#1}{?}}
     { \gdef\shortauthor{\let\comma=\empty \let\corrauth=\empty \renewcommand{\affil}[1]{} #2} }
     { \gdef\shortauthor{#1}}
}
\def\@setauthor{\begin{center}
  \sc  \@author
    \end{center}%
}

\makeatother
\title{\bf\Large
Characterization of BMO spaces via commutators of some maximal operators on the slice spaces }
\author[J. WU and W. Zhang] { Yunpeng Chang\affil{1}, Jianglong Wu\affil{1}\comma\corrauth, Yida Sun\affil{1}
\\ 
{\affilnum{1}\footnotesize\it Department of Mathematics, Mudanjiang Normal University, Mudanjiang, 157011, China}
}
\date{} 
\begin{document}

\maketitle


\footnote{\textit{Email address}:\  {\bf jl-wu@163.com}\\
$~~^{~~~~}$Chang and Wu have contributed equally to this work and should be considered co-first authors.}  
\vspace{-4em}
\renewcommand\abstractname{}
\begin{abstract}
\setlength{\parindent}{0pt}\setlength{\parskip}{1.5explus0.5exminus0.2ex}%
\noindent
{\textbf{Abstract}:} In this paper, the main aim  is to demonstrate the boundedness for commutators of (fractional) maximal function and sharp maximal function in the slice spaces, where the symbols of the commutators belong to the $\BMO$ space, whereby some new characterizations for $\BMO$ spaces are given.


 { \bf AMS(2020) Subject Classification:}  \ 26A33, 42B20, 42B25\ \
\end{abstract}

\section{Introduction and main results}
\noindent Let $T$ be the classic singular integral operator, the Coifman-Rochberg-Weiss type commutator $[b,T]$ generated by $T$ and a suitable function $b$ is defined by
$$[b,T]f=bT(f)-T(bf).$$

An important conclusion is that $b\in \BMO(\mathbb{R}^{n})$ if and only if $[b,T]$ is bounded on $L^{s}(\mathbb{R}^{n})$ for $1<s<\infty,$ we refer to see Coifman, Rochberg and Weiss \cite{coifman1976factorization}(see also \cite{janson1978mean}).

It is worth mentioning that the maximal operator is the most basic and important class of operators in harmonic analysis. Up to now, this class of operators has not only been studied by many people and extended to different fields, such as Lie group theory, number theory etc \cite{wu2023some,wu2024characterization,wu2024some}, but also applied to mathematical physics, partial differential equations and other branches\cite{tao2000electro,nolder1991hardy}. Moreover, we first recall some operators (commutators can also be understood as an operator).

Let $0<\alpha<n,$ the fractional maximal operator can be defined by
 \begin{align*}
 M_{\alpha}(f)(x)=\sup_{Q\ni x}\frac{1}{|Q|^{1-\frac{\alpha}{n}}}\int_{Q}|f(y)|dy,
 \end{align*}
 where the supremum is taken over all cubes $Q\subset \mathbb{R}^{n}.$ For $\alpha=0,$~$M_{0}=M$(classic Hardy-Littlewood maximal operator).

 If $b\in L_{\loc}^{1}(\mathbb{R}^{n})$, the commutator produced by $b$~with $M_{\alpha}$ is given by
  \begin{align*}
  M_{\alpha,b}(f)(x)=\sup_{Q\ni x}\frac{1}{|Q|^{1-\frac{\alpha}{n}}}\int_{Q}|b(x)-b(y)||f(y)|dy,
  \end{align*}
  where the supremum is taken over all cubes $Q\subset \mathbb{R}^{n}.$  For $\alpha=0,$~$M_{0,b}=M_{b}$(maximal commutator).~And the nonlinear commutators of $M_{\alpha}$ can be defined by
\begin{align*}
[b,M_{\alpha}]f(x)=b(x)M_{\alpha}(f)(x)-M_{\alpha}(bf)(x).
\end{align*}
For $\alpha=0,$ we write $[b,M]=[b,M_{0}].$

Fefferman and Stein \cite{fefferman1972h} introduced the sharp maximal function, defined as
 \begin{align*}
 M^{\sharp}(f)(x)=\sup_{Q\ni x}\frac{1}{|Q|}\int_{Q}|f(y)-f_{Q}|dy,
 \end{align*}
 where the supremum is taken over all cubes $Q\subset \mathbb{R}^{n},$~$f_{Q}$ is the average of $f$ over $Q.$~And the nonlinear commutator produced by $b$~with $ M^{\sharp}$ is provided by
\begin{align*}
[b,M^{\sharp}]f(x)=b(x)M^{\sharp}(f)(x)-M^{\sharp}(bf)(x).
\end{align*}

Noting that $M_{\alpha,b}$ and $[b,M_{\alpha}]$ essentially differ from each other. For instant, $M_{\alpha,b}$ is positive and sublinear, however, $[b,M_{\alpha}]([b,M^{\sharp}])$ is neither positive nor sublinear.

 When $b\in \BMO(\mathbb{R}^{n}),$ Bastero et al. \cite{bastero2000commutators} studied the boundedness of $[b,M]$ in $L^{q}(\mathbb{R}^{n})$ for $1<q<\infty.$ In \cite{zhang2009commutators,zhang2014commutators,zhang2014commutators1} by Zhang and Wu further obtained the boundedness for $M_{\alpha,b},~[b,M_{\alpha}],[b,M^{\sharp}]$ on variable Lebesgue spaces. Guliyev \cite{guliyev2022some} obtained the boundedness for  $M_{\alpha,b},~[b,M_{\alpha}]$ in the Orlicz space on any stratified Lie group. However,~the study of operators on slice spaces seems to be very little.

Inspired by the above literatures, the purpose of this paper is to study the the boundedness for commutators of fractional maximal function and sharp maximal function in the slice spaces, where the symbols of the commutators belong to the $\BMO$ spaces, as a corollary, we also obtain the boundedness of maximal commutator in the slice spaces.


Let $\alpha \geq 0,$ for a fixed cube $Q_{*},$ the fractional maximal function with respect to $Q_{*}$ of locally integrable function $f$ is given by
\begin{align*}
M_{\alpha,Q_{*}}(f)(x)=\sup_{\stackrel{Q\ni x}{Q\subset Q_{*}}}\frac{1}{|Q|^{1-\frac{\alpha}{n}}}\int_{Q}|f(y)|dy,
\end{align*}
where the supremum is taken over all cubes $Q$ with $Q\subset Q_{*}$. If $\alpha=0, M_{Q_{*}}=M_{0,Q_{*}}.$

In order to introduce the following theorem, the $b^{-}$ is defined by
\begin{align*}
b^{-}(x)
= \left\{\begin{matrix}
0,\ &\mathrm{if}\ b(x)\geq 0,\\
|b(x)|,\ &\mathrm{if}\ b(x)<0,\\
\end{matrix}\right.
~~\mathrm{and}~~ b^{+}(x)=|b(x)|-b^{-}(x).
\end{align*}

First of all, we consider the boundedness of nonlinear commutators generated by the fractional maximal function and BMO function on slice spaces, and establish some new equivalent characterizations of BMO spaces which are different from Morrey spaces, (variable) Lebesgue spaces and Orlicz spaces.

\begin{theorem}\label{thm.1-0}Let $0<\alpha<n$ and $b$~be~a locally integrable function on $\mathbb{R}^{n}.$ If $1<p<r<\infty,~1<q<s<\infty$~and $\alpha/n=1/p-1/r=1/q-1/s,$~then the following statements are equivalent.

(T.1) $b\in \mathrm{BMO}(\mathbb{R}^{n})$ and $b^{-}\in L^{\infty}(\mathbb{R}^{n}).$

(T.2) $[b,M_{\alpha}]$ is bounded from $(E_{p}^{q})_{t}(\mathbb{R}^{n})~\mathrm{to}~(E_{r}^{s})_{t}(\mathbb{R}^{n}).$

(T.3) There exists a constant $C>0$ such that
\begin{align*}
\sup_{Q}\frac{1}{|Q|^{1/s}}\left\|(b-|Q|^{-\alpha/n}M_{\alpha,Q}(b))\chi_{Q}\right\|_{(E_{r}^{s})_{t}(\mathbb{R}^{n})}\leq C.
\end{align*}

(T.4) There exists a constant $C>0$ such that
\begin{align*}
\sup_{Q}\frac{1}{|Q|}\int_{Q}|b(y)-M_{Q}(b)(y)|dy\leq C.
\end{align*}
\end{theorem}

When we take $\alpha=0$ and make a slight modification of the above theorem, it is not difficult to obtain the boundedness of nonlinear commutators generated by Hardy-Littlewood maximal functions and BMO functions in slice spaces, thus establishing a new equivalent characterization different from Theorem 1.1.~In fact, the proof is almost similar to the theorem mentioned above. For this we present the conclusion below and it is even new.

\begin{corollary}
Let $b$~be~a locally integrable function on $\mathbb{R}^{n}.$ If $1<p<\infty,~1<q<\infty$,~then the following statements are equivalent.

(C.1) $b\in \mathrm{BMO}(\mathbb{R}^{n})$ and $b^{-}\in L^{\infty}(\mathbb{R}^{n}).$

(C.2) $[b,M]$ is bounded on $(E_{p}^{q})_{t}(\mathbb{R}^{n}).$

(C.3) There exists a constant $C>0$ such that
\begin{align*}
\sup_{Q}\frac{1}{|Q|}\left\|(b-M_{Q}(b))\chi_{Q}\right\|^{q}_{(E_{p}^{q})_{t}(\mathbb{R}^{n})}\leq C.
\end{align*}

(C.4) There exists a constant $C>0$ such that
\begin{align*}
\sup_{Q}\frac{1}{|Q|}\int_{Q}|b(y)-M_{Q}(b)(y)|dy\leq C.
\end{align*}
\end{corollary}

Secondly,~we obtain the boundedness of commutators generated by the fractional maximal function and BMO function on slice spaces, and establish some new equivalent characterizations of BMO spaces as follows.

\begin{theorem}\label{thm.2-0}
Let $0<\alpha<n$ and $b$~be~a locally integrable function on $\mathbb{R}^{n}.$ If $1<p<r<\infty,~1<q<s<\infty$~and $\alpha/n=1/p-1/r=1/q-1/s,$~then the following statements are equivalent.

(T.1) $b\in \mathrm{BMO}(\mathbb{R}^{n})$.

(T.2) $M_{\alpha,b}$ is bounded from $(E_{p}^{q})_{t}(\mathbb{R}^{n})~\mathrm{to}~(E_{r}^{s})_{t}(\mathbb{R}^{n}).$

(T.3) There exists a constant $C>0$ such that
\begin{align*}
\sup_{Q}\frac{1}{|Q|^{1/s}}\left\|(b-b_{Q})\chi_{Q}\right\|_{(E_{r}^{s})_{t}(\mathbb{R}^{n})}\leq C.
\end{align*}

(T.4) There exists a constant $C>0$ such that
\begin{align*}
\sup_{Q}\frac{1}{|Q|}\int_{Q}|b(y)-b_{Q}|dy\leq C.
\end{align*}
\end{theorem}

In analogy with corollary 1.1, for the case of $\alpha=0$ and also make a slight modification of the above theorem, it is easy to get the boundedness of commutators generated by Hardy-Littlewood maximal functions and BMO functions in slice spaces, by which a new equivalent characterization different from Theorem 1.2 is obtained.~Of course, the proof is almost similar to the theorem mentioned above. For this we present the conclusion below and it is even new.

\begin{corollary}
Let $b$~be~a locally integrable function on $\mathbb{R}^{n}.$ If $1<p<\infty,~1<q<\infty$,~then the following statements are equivalent.

(C.1) $b\in \mathrm{BMO}(\mathbb{R}^{n}).$

(C.2) $M_{b}$ is bounded on $(E_{p}^{q})_{t}(\mathbb{R}^{n}).$

(C.3) There exists a constant $C>0$ such that
\begin{align*}
\sup_{Q}\frac{1}{|Q|}\left\|(b-b_{Q})\chi_{Q}\right\|^{q}_{(E_{p}^{q})_{t}(\mathbb{R}^{n})}\leq C.
\end{align*}

(C.4) There exists a constant $C>0$ such that
\begin{align*}
\sup_{Q}\frac{1}{|Q|}\int_{Q}|b(y)-b_{Q}|dy\leq C.
\end{align*}
\end{corollary}

Finally,~we study the boundedness of nonlinear commutators generated by the sharp maximal function and BMO function on slice spaces, and establish some new equivalent characterizations of BMO spaces.

\begin{theorem}\label{thm.3-0}
Let $b$~be~a locally integrable function on $\mathbb{R}^{n}.$ If $1<p<\infty,~1<q<\infty$,~then the following statements are equivalent.

(T.1) $b\in \mathrm{BMO}(\mathbb{R}^{n})$~and~$b^{-}\in L^{\infty}(\mathbb{R}^{n}).$

(T.2) $[b,M^{\sharp}]$ is bounded on $(E_{p}^{q})_{t}(\mathbb{R}^{n}).$

(T.3) There exists a constant $C>0$ such that
\begin{align*}
\sup_{Q}\frac{1}{|Q|}\left\|(b-2M^{\sharp}(b\chi_{Q}))\chi_{Q}\right\|^{q}_{(E_{p}^{q})_{t}(\mathbb{R}^{n})}\leq C.
\end{align*}

(T.4) There exists a constant $C>0$ such that
\begin{align*}
\sup_{Q}\frac{1}{|Q|}\int_{Q}|b(y)-2M^{\sharp}(b\chi_{Q})(y)|dy\leq C.
\end{align*}
\end{theorem}

  Throughout this paper, the letter $C$ always takes place of a constant independent of the primary parameters involved and whose value may differ from line to line. In addition, we give some notations. Here and hereafter $|E|$ will always denote the Lebesgue measure of a measurable set $E$ on $\mathbb{R}^{n}$ and by $\chi_{E}$ denotes the characteristic function of a measurable set $E \subset \mathbb{R}^{n}.$

\section{Preliminaries}

\subsection{Some function spaces}
\noindent This subsection introduces some main function spaces,~for example, slice,~BMO spaces.

In the last years, the research on slice spaces has attracted considerable attention. In 2019, the earliest lesson on slice spaces~$(E _{2}^{p})_{t}(\mathbb{R}^{n})(0<t<\infty,~1<p<\infty)$~can be traced back to the work of Auscher and Mourgolou in \cite{auscher2019representation}, for the purpose of the weak solutions of boundary value problems with a $t$-independent elliptic systems in the upper half plane. Recently, Auscher and Prisuelos-Arribas \cite{auscher2017tent} considered the boundedness of some classical operators on the slice space  $(E^{p}_{r})_{t}(\mathbb{R}^{n})(0<t<\infty$~and~$1<r,~p<\infty)$. Of course, for more research on slice spaces, we refer to see \cite{lu2022necessary,heng2023some}.
 \begin{definition}
 Let $0<t<\infty$~and~$1<r,~p<\infty.$ The slice space $(E^{p}_{r})_{t}(\mathbb{R}^{n})$ is defined as the set of all locally $r$-integrable functions $f$ on $\mathbb{R}^{n}$ such that
$$\|f\|_{(E^{p}_{r})_{t}(\mathbb{R}^{n})}=\left(\int_{\mathbb{R}^{n}}\left(\frac{1}{|Q(x,t)|}\int_{Q(x,t)}|f(y)|^{r}dy\right)^{p/r}\right)^{1/p}<\infty.$$
 \end{definition}

John and Nirenberg \cite{john1961functions} introduced the following space for the purpose of studying the partial differential equations problem.

\begin{definition}
Let $f\in L_{\loc}^{1}(\mathbb{R}^{n}),$~then the bounded mean oscillation (BMO) space can be defined by
\begin{align*}
\BMO(\mathbb{R}^{n})=\left\{f\in L_{\loc}^{1}(\mathbb{R}^{n}):\sup_{Q}\frac{1}{|Q|}\int_{Q}|f(y)-f_{Q}|dy<\infty\right\},
\end{align*}
where~$f_{Q}$ is the average of $f$ over $Q.$
\end{definition}
\subsection{Auxiliary propositions  and lemmas }
In this part we state some auxiliary propositions and lemmas which will be needed for proving our main theorems. And we only describe partial results we need.

First, we introduce the well-known H$\mathrm{\ddot{o}}$lder's inequality,~which plays an extensive role in the proof of this paper.
\begin{lemma}
Let~$1<p<\infty$~and~$p'$ is the conjugate exponent of $p.$~If $f\in L^{p}(\mathbb{R}^{n})~ \mathrm{and} ~g\in L^{p'}(\mathbb{R}^{n}).$ Then
$$\int_{\mathbb{R}^{n}}|f(x)g(x)|dx\leq C\|f\|_{L^{p}(\mathbb{R}^{n})}\|g\|_{L^{p'}(\mathbb{R}^{n})}.$$

\end{lemma}

The following Lemmas 2.2 and 2.3 can be obtained in \cite{lu2022necessary}, where the Lemma 2.2 reveals the boundedness of the maximal function in the slice space, and Lemma 2.3 gives the norm of the characteristic function on the slice space, which is different from the norm of the characteristic function on the (variable) Lebesgue and Morrey spaces.

\begin{lemma}\label{lem.variable-proposition-B}
 Let $0<t<\infty$~and~$1<r,p<\infty.$~If~$f\in (E_{r}^{p})_{t}(\mathbb{R}^{n}),$~then
\begin{align*}
\|Mf\|_{(E_{r}^{p})_{t}(\mathbb{R}^{n})}\leq C\|f\|_{(E_{r}^{p})_{t}(\mathbb{R}^{n})}.
\end{align*}
\end{lemma}

\begin{lemma}
If $0<t<\infty$~and~$1<r,p<\infty,$~then for any cube $Q\subset \mathbb{R}^{n},$ we have
 $$\|\chi_{Q}\|_{(E_{r}^{p})_{t}(\mathbb{R}^{n})}\approx |Q|^{1/p}.$$
\end{lemma}

The following result shows the boundedness of the fractional maximal operator in slice spaces, which plays an important role in the estimates of commutators. For some details, we can see \cite{lu2022some}.

\begin{lemma}
 Let $0<\alpha<n,~0<t<\infty,~1<r,p<\infty$~and~$1<q<s<\infty$~with~$\alpha/n=1/p-1/r=1/q-1/s.$~If $f\in (E_{p}^{q})_{t}(\mathbb{R}^{n}),$~then
 \begin{align*}
\|M_{\alpha}f\|_{(E_{r}^{s})_{t}(\mathbb{R}^{n})}\leq C\|f\|_{(E_{p}^{q})_{t}(\mathbb{R}^{n})}.
\end{align*}
\end{lemma}

The following estimate can be found in \cite{guliyev2021commutators}.
\begin{lemma}
Let $0<\alpha<n$, if $b\in \BMO(\mathbb{Q}_{p}^{n}),$ then there exists a constant $C>0,$ such that for almost every $x\in \mathbb{R}^{n}$ and for all functions $f\in L_{\mathrm{loc}}^{1}(\mathbb{R}^{n}),$ we have
\begin{align*}
 M_{\alpha,b}(f)(x)\leq C\|b\|_{\ast}(M(M_{\alpha}f)(x)+M_{\alpha}(Mf)(x)).
\end{align*}
\end{lemma}

\begin{lemma}\cite{zhang2009commutators}
Let $b\in L^{1}_{loc}(\mathbb{R}^{n}).$ For any fixed cube $Q\subset \mathbb{R}^{n}.$

(1) If $0\leq \alpha<n,$ then for all $y\in Q,$ we have
$$M_{\alpha}(b\chi_{Q})(y)=M_{\alpha,Q}(b)(y)$$
and
$$M_{\alpha}(\chi_{Q})(y)=M_{\alpha,Q}(\chi_{Q} )(y)=|Q|^{\frac{\alpha}{n}}.$$

(2) Then for any $y\in Q,$ we have
$$|b_{Q}|\leq |Q|^{-\frac{\alpha}{n}}M_{\alpha,Q}(b)(y). $$

(3) Let $E=\{y\in Q:b(y)\leq b_{Q}\}$ and $F=Q\setminus E=\{y\in Q:b(y)> b_{Q}\}.$ Then the following equality is trivially true
$$\int_{E}|b(y)-b_{Q}|dy=\int_{F}|b(y)-b_{Q}|dy. $$
\end{lemma}

Bastero et al. \cite{bastero2000commutators} obtained the following equivalent relations for BMO space.

\begin{lemma}
Let b be a locally integrable function on $\mathbb{R}^{n}.$ Then the following assertions are equivalent:

(1) $b\in \BMO(\mathbb{R}^{n})$ and $b^{-}\in L^{\infty}(\mathbb{R}^{n}).$

(2) There exists a positive constant $C$ such that
\begin{align*}
\sup_{Q}\frac{1}{|Q|}\int_{Q}|b(y)-M_{Q}(b)(y)|dy\leq C.
\end{align*}

(3)  There exists a positive constant $C$ such that
\begin{align*}
\sup_{Q}\frac{1}{|Q|}\int_{Q}|b(y)-2M^{\sharp}(b\chi_{Q})(y)|dy\leq C.
\end{align*}
\end{lemma}

\section{Proof of the principal results}
\begin{lemma}Let $0<\alpha<n$ and $b$~be~a locally integrable function on $\mathbb{R}^{n}.$ If there exists a constant $C>0$ such that
\begin{align*}
\sup_{Q}\frac{1}{|Q|^{1/s}}\left\|(b-|Q|^{-\alpha/n}M_{\alpha,Q}(b))\chi_{Q}\right\|_{(E_{r}^{s})_{t}(\mathbb{R}^{n})}\leq C
\end{align*}
for $1<r,~s<\infty,$~
then $b\in \BMO(\mathbb{R}^{n}).$
\end{lemma}
\begin{proof}
For any fixed cube $Q\subset \mathbb{R}^{n},$ according to Lemma 2.6 (2), if $y\in Q,$ then
\begin{align*}
|b_{Q}|\leq |Q|^{-\frac{\alpha}{n}}M_{\alpha,Q}(b)(y).
\end{align*}
Let $E=\{y\in Q:b(y)\leq b_{Q}\}$ and $F=Q\setminus E=\{y\in Q:b(y)> b_{Q}\},$ then for any $y\in E,$ we have $b(y)\leq b_{Q}\leq |b_{Q}|\leq |Q|^{-\frac{\alpha}{n}}M_{\alpha,Q}(b)(y),$ which implies
$$|b-b_{Q}|\leq |b-|Q|^{-\alpha/n}M_{\alpha,Q}(b)(y)|.$$
Thus applying Lemma 2.6 (3), we obtain
\begin{align*}
\frac{1}{|Q|}\int_{Q}|b(y)-b_{Q}|dy&=\frac{1}{|Q|}\int_{E\cup F}|b(y)-b_{Q}|dy\\
&=\frac{2}{|Q|}\int_{E}|b(y)-b_{Q}|dy\\
&\leq \frac{2}{|Q|}\int_{E}|b(y)-|Q|^{-\alpha/n}M_{\alpha,Q}(b)(y)|dy\\
&\leq \frac{2}{|Q|}\int_{Q}|b(y)-|Q|^{-\alpha/n}M_{\alpha,Q}(b)(y)|dy.
\end{align*}
By using Lemmas 2.1,~2.3 and the condition, we get
\begin{align*}
\frac{1}{|Q|}\int_{Q}|b(y)-b_{Q}|dy&\leq 2\left(\frac{1}{|Q|}\int_{Q}|b(y)-|Q|^{-\alpha/n}M_{\alpha,Q}(b)(y)|^{r}dy\right)^{1/r}\left(\frac{1}{|Q|}\int_{Q}|\chi_{Q}|^{r'}dy\right)^{1/r'}\\
&\leq \frac{C}{|Q|}\|(b-|Q|^{-\alpha/n}M_{\alpha,Q}(b))\chi_{Q}\|_{(E_{r}^{s})_{t}(\mathbb{R}^{n})}\|\chi_{Q}\|_{(E_{r'}^{s'})_{t}(\mathbb{R}^{n})}\\
& \leq  C.
\end{align*}
Which implies $b\in \BMO(\mathbb{R}^{n}).$ Thus we finish the proof of Lemma 3.1.
\end{proof}
\begin{refproof}[Proof of \cref{thm.1-0}]
Since $(T.4)\Rightarrow (T.1)$ follow readily from Lemma 2.7, we only need to prove $(T.1)\Rightarrow (T.2),~(T.2)\Rightarrow (T.3)$ and $(T.3)\Rightarrow (T.4).$

$(T.1)\Rightarrow (T.2):$ Let $b\in \mathrm{BMO}(\mathbb{R}^{n})$ and $b^{-}\in L^{\infty}(\mathbb{R}^{n}).$~We need to prove $[b,M_{\alpha}]$ is bounded from $(E_{p}^{q})_{t}(\mathbb{R}^{n})~\mathrm{to}~(E_{r}^{s})_{t}(\mathbb{R}^{n}).$ It follows from Lemma 2.4 that $M_{\alpha}f(x)<\infty(\mathrm{when} f\in(E_{p}^{q})_{t}(\mathbb{R}^{n}))$ for almost $x\in \mathbb{R}^{n},$ we have (for example see \cite{zhang2014commutators1})
\begin{align}
\left|[b,M_{\alpha}]f(x)\right|\leq M_{\alpha,b}(f)(x)+2b^{-}M_{\alpha}(f)(x).
\end{align}
Using Lemma 2.5, we further obtain
\begin{align*}
\left|[b,M_{\alpha}]f(x)\right|\leq  C\|b\|_{\ast}(M(M_{\alpha}f)(x)+M_{\alpha}(Mf)(x))+2b^{-}M_{\alpha}(f)(x).
\end{align*}
With the help of Lemma 2.2 and Lemma 2.4, we get
\begin{align*}
\|[b,M_{\alpha}]f\|_{(E_{r}^{s})_{t}(\mathbb{R}^{n})}&\leq C\|b\|_{\ast}\|M(M_{\alpha}f)+M_{\alpha}(Mf))\|_{(E_{r}^{s})_{t}(\mathbb{R}^{n})}+2\|b^{-}\|_{L^{\infty}(\mathbb{R}^{n})}\|M_{\alpha}(f)\|_{(E_{r}^{s})_{t}(\mathbb{R}^{n})}\\
&\leq \|M_{\alpha}f\|_{(E_{r}^{s})_{t}(\mathbb{R}^{n})}+\|Mf\|_{(E_{p}^{q})_{t}(\mathbb{R}^{n})}\\
&\leq C\|f\|_{(E_{p}^{q})_{t}(\mathbb{R}^{n})}.
\end{align*}
Thus, we give that $[b,M_{\alpha}]$ is bounded from $(E_{p}^{q})_{t}(\mathbb{R}^{n})~\mathrm{to}~(E_{r}^{s})_{t}(\mathbb{R}^{n}).$

$(T.2)\Rightarrow (T.3):$ For any fixed cube $Q\subset \mathbb{R}^{n}$ and $y\in Q,$~it follows from Lemma 2.6 (1) that
$$M_{\alpha}(b\chi_{Q})(y)=M_{\alpha,Q}(b)(y)
~~\mathrm{and}~~M_{\alpha}(\chi_{Q})(y)=M_{\alpha,Q}(\chi_{Q} )(y)=|Q|^{\frac{\alpha}{n}}.$$
Then for any $y\in Q,$~we have
\begin{align*}
b(y)-|Q|^{-\alpha/n}M_{\alpha,Q}(b)(y)&=|Q|^{-\alpha/n}(b(y)|Q|^{\alpha/n}-M_{\alpha,Q}(\chi_{Q} )(y))\\
&=|Q|^{-\alpha/n}(b(y)M_{\alpha}(\chi_{Q})(y)-M_{\alpha}(b\chi_{Q})(y))\\
&=|Q|^{-\alpha/n}[b,M_{\alpha}](\chi_{Q})(y).
\end{align*}
Then for any  $y\in \mathbb{R}^{n},$ we obtain
\begin{align*}
b(y)-|Q|^{-\alpha/n}M_{\alpha,Q}(b)(y)\chi_{Q}(y)=|Q|^{-\alpha/n}[b,M_{\alpha}](\chi_{Q})(y).
\end{align*}
By applying Lemma 2.3,~statement (T.2) and the condition $\alpha/n=1/q-1/s,$ we get
\begin{align*}
\left\|(b-|Q|^{-\alpha/n}M_{\alpha,Q}(b))\chi_{Q}\right\|_{(E_{r}^{s})_{t}(\mathbb{R}^{n})}&=|Q|^{-\alpha/n}\left\|[b,M_{\alpha}](\chi_{Q}) \right\|_{(E_{r}^{s})_{t}(\mathbb{R}^{n})}\\
&\leq C|Q|^{-\alpha/n}\left\|\chi_{Q} \right\|_{(E_{p}^{q})_{t}(\mathbb{R}^{n})}\\
&\leq C|Q|^{-\alpha/n+1/q}\\
&=C|Q|^{1/s},
\end{align*}
where the arbitrary constant $C$ is independent of $Q.$ Thus we deduce (T.3).

$(T.3)\Rightarrow (T.4):$  For any fixed cube $Q\subset \mathbb{R}^{n},$~we have
\begin{align*}
\frac{1}{|Q|}\int_{Q}|b(y)-M_{Q}(b)(y)|dy&\leq \frac{1}{|Q|}\int_{Q}|b(y)-|Q|^{-\alpha/n}M_{\alpha,Q}(b)(y)|dy\\
&~~+\frac{1}{|Q|}\int_{Q}||Q|^{-\alpha/n}M_{\alpha,Q}(b)(y)-M_{Q}(b)(y)|dy\\
&=I_{1}+I_{2}.
\end{align*}
For the first term $I_{1},$~by using Lemma 2.3, statement (T.3) and Lemma 2.1, we obtain
\begin{align*}
I_{1}&\leq \left(\frac{1}{|Q|}\int_{Q}|b(y)-|Q|^{-\alpha/n}M_{\alpha,Q}(b)(y)|^{r}dy\right)^{1/r}\left(\frac{1}{|Q|}\int_{Q}|\chi_{Q}|^{r'}dy\right)^{1/r'}\\
&\leq \frac{C}{|Q|}\|(b-|Q|^{-\alpha/n}M_{\alpha,Q}(b))\chi_{Q}\|_{(E_{r}^{s})_{t}(\mathbb{R}^{n})}\|\chi_{Q}\|_{(E_{r'}^{s'})_{t}(\mathbb{R}^{n})}\\
& \leq  C.
\end{align*}
Next, for all $y\in Q,$ it follows from Lemma 2.6 (1) that
$$M_{\alpha}(b\chi_{Q})(y)=M_{\alpha,Q}(b)(y)
~~\mathrm{and}~~ M_{\alpha}(\chi_{Q})(y)=|Q|^{\frac{\alpha}{n}},$$
and
$$M(b\chi_{Q})(y)=M_{Q}(b)(y)
~~\mathrm{and}~~ M(\chi_{Q})(y)=1.$$
Then for all $y\in Q,$
 \begin{align*}
&\left||Q|^{-\alpha/n}M_{\alpha,Q}(b)(y)-M_{Q}(b)(y)\right|\\
&\leq |Q|^{-\alpha/n}\left|M_{\alpha,Q}(b)(y)-|Q|^{\alpha/n}|b(y)|\right|+\left||b(y)|-M_{Q}(b)(y)\right|\\
&\leq |Q|^{-\alpha/n}\left|M_{\alpha}(b\chi_{Q})(y)-M_{\alpha}(\chi_{Q})(y)|b(y)|\right|+\left||b(y)|M(\chi_{Q})(y)-M(b\chi_{Q})(y)\right|\\
&\leq |Q|^{-\alpha/n}\left|[|b|,M_{\alpha}](\chi_{Q})(y)\right|+\left|[|b|,M](\chi_{Q})(y)\right|.
\end{align*}
The statement (T.3) along with Lemma 3.1 gives $b\in \BMO(\mathbb{R}^{n}),$ which implies $|b|\in \BMO(\mathbb{R}^{n}).$
It follows from (3.1), Lemma 2.5 and Lemma 2.6 (1) that
\begin{align*}
\left|[|b|,M_{\alpha}](\chi_{Q})(y)\right|&\leq M_{\alpha,b}(\chi_{Q})(y)\leq C\|b\|_{\ast}(M(M_{\alpha}(\chi_{Q}))(y)+M_{\alpha}(M(\chi_{Q}))(y))\\
&\leq C\|b\|_{\ast}|Q|^{\alpha/n}.
\end{align*}
And
\begin{align*}
\left|[|b|,M](\chi_{Q})(y)\right|&\leq M_{|b|}(\chi_{Q})(y)\leq C\|b\|_{\ast}M(M_{\alpha}(\chi_{Q}))(y)\leq C\|b\|_{\ast}.
\end{align*}
Thus we obtain
\begin{align*}
I_{2}\leq \frac{1}{|Q|}\int_{Q}|Q|^{-\alpha/n}\left|[|b|,M_{\alpha}](\chi_{Q})(y)\right|+\left|[|b|,M](\chi_{Q})(y)\right|dy\leq  C\|b\|_{\ast}.
\end{align*}
Combining with $I_{1}$ and $I_{2},$~we have
\begin{align*}
\sup_{Q}\frac{1}{|Q|}\int_{Q}|b(y)-M_{Q}(b)(y)|dy\leq C.
\end{align*}
Therefore we finish the proof of Theorem 1.1.
\end{refproof}

\begin{refproof}[Proof of \cref{thm.2-0}]
Since the statements $(T.1)\Longleftrightarrow (T.4)$ directly are obtained by Definition 2.2, we only need to prove $(T.1)\Longrightarrow(T.2)$, $(T.2)\Longrightarrow (T.3)$ and $(T.3)\Longrightarrow (T.4).$

$(T.1)\Longrightarrow (T.2).$ Since $b\in \BMO(\mathbb{R}^{n}),$ by using Lemma 2.5, we obtain
\begin{align*}
 M_{\alpha,b}(f)(x)\leq C\|b\|_{\ast}(M(M_{\alpha}f)(x)+M_{\alpha}(Mf)(x)).
\end{align*}
Similar to $(T.1)\Longrightarrow (T.2)$ of Theorem 1.1. It follows from Lemma 2.2 and Lemma 2.4 that  $M_{\alpha,b}$ is bounded from $(E_{p}^{q})_{t}(\mathbb{R}^{n})~\mathrm{to}~(E_{r}^{s})_{t}(\mathbb{R}^{n}).$

$(T.2)\Longrightarrow (T.3).$ For any fixed cube $Q\subset \mathbb{R}^{n}$ and $y\in Q,$ we have
\begin{align*}
|b(y)-b_{Q}|&=\frac{1}{|Q|}\int_{Q}|b(y)-b(z)|dz\\
&=\frac{1}{|Q|}\int_{Q}|b(y)-b(z)|\chi_{Q}(z)dz\\
&=\frac{1}{|Q|^{\alpha/n}}M_{\alpha,b}(\chi_{Q})(y).
\end{align*}
Then for any  $y\in Q,$ we obtain
\begin{align*}
|(b(y)-b_{Q})\chi_{Q}(y)|=|Q|^{-\alpha/n}M_{\alpha,b}(\chi_{Q})(y).
\end{align*}
By applying Lemma 2.3,~statement (T.2) and the condition $\alpha/n=1/q-1/s,$ we get
\begin{align*}
\left\|(b-b_{Q})\chi_{Q}\right\|_{(E_{r}^{s})_{t}(\mathbb{R}^{n})}&=|Q|^{-\alpha/n}\left\|M_{\alpha,b}(\chi_{Q}) \right\|_{(E_{r}^{s})_{t}(\mathbb{R}^{n})}\\
&\leq C|Q|^{-\alpha/n}\left\|\chi_{Q} \right\|_{(E_{p}^{q})_{t}(\mathbb{R}^{n})}\\
&\leq C|Q|^{-\alpha/n+1/q}=C|Q|^{1/s},
\end{align*}
where the arbitrary constant $C$ is independent of $Q.$ Thus we deduce (T.3).

$(T.3)\Longrightarrow (T.4).$ For any fixed cube $Q\subset \mathbb{R}^{n},$~using Lemmas 2.1, 2.3 and statement (T.3) again, we have
\begin{align*}
\frac{1}{|Q|}\int_{Q}|b(y)-b_{Q}|dy&\leq 2\left(\frac{1}{|Q|}\int_{Q}|b(y)-b_{Q}|^{r}dy\right)^{1/r}\left(\frac{1}{|Q|}\int_{Q}|\chi_{Q}|^{r'}dy\right)^{1/r'}\\
&\leq \frac{C}{|Q|}\|(b-b_{Q})\chi_{Q}\|_{(E_{r}^{s})_{t}(\mathbb{R}^{n})}\|\chi_{Q}\|_{(E_{r'}^{s'})_{t}(\mathbb{R}^{n})} \leq  C.
\end{align*}

Thus the proof of Theorem 1.2 is completed.
\end{refproof}

\begin{refproof}[Proof of \cref{thm.3-0}]
Since the implications $(T.1)\Longleftrightarrow(T.4)$ follow readily from Lemma 2.7, we only need to prove $(T.1)\Longrightarrow(T.2), (T.2)\Longrightarrow(T.3),(T.3)\Longrightarrow(T.4).$

$(T.1)\Longrightarrow(T.2).$ Assume $b\in \BMO(\mathbb{R}^{n})$ and $b\in L^{\infty}(\mathbb{R}^{n}),$ for any fixed cube $Q\subset \mathbb{R}^{n},$ the following estimate was obtained in \cite{zhang2014commutators}:
\begin{align*}
|[|b|,M^{\sharp}]f(x)|\leq 2M_{|b|}f(x).
\end{align*}
Since~$M^{\sharp}(f)\leq 2M(f)$, for any~$x\in\mathbb{R}^{n},$ we have
\begin{align*}
|[b,M^{\sharp}](f)(x)|&\leq 2((b^{-}(x))M^{\sharp}(f)(x)+M^{\sharp}_{p}(b^{-}f)(x))+|[|b|,M^{\sharp}](f)(x)|\\
&\leq 4((b^{-}(x))M(f)(x)+M(b^{-}f)(x))+2M_{|b|}f(x).
\end{align*}
 Since $b\in \BMO(\mathbb{R}^{n}),$ then $|b|\in \BMO(\mathbb{R}^{n}).$ It follows from Corollary 1.2 and Lemma 2.2 that
\begin{align*}
\|[b,M^{\sharp}](f)\|_{(E_{p}^{q})_{t}(\mathbb{R}^{n})}\leq C\|b\|_{\ast}\|f\|_{(E_{p}^{q})_{t}(\mathbb{R}^{n})}.
\end{align*}
Which implies that $[b,M^{\sharp}]$ is bounded on $(E_{p}^{q})_{t}(\mathbb{R}^{n}).$

$(T.2)\Longrightarrow(T.3).$ Assume $[b,M^{\sharp}]$ is bounded on $(E_{p}^{q})_{t}(\mathbb{R}^{n}).$ For any fixed cube $Q,$ we have (see also \cite{zhang2014commutators})
$$M^{\sharp}(\chi_{Q})(y)=\frac{1}{2},~~~y\in Q.$$
Then, for all $y\in Q,$ we obtain
\begin{align*}
&b(y)-2M^{\sharp}(b\chi_{Q})(y)=2(\frac{1}{2}b(y)-M^{\sharp}(b\chi_{Q})(y))\\
&=2(b(y)M^{\sharp}(\chi_{Q})(y)-M^{\sharp}(b\chi_{Q})(y))=2[b,M^{\sharp}](\chi_{Q})(y).
\end{align*}
Then for any  $y\in \mathbb{R}^{n},$ we obtain
\begin{align*}
&(b(y)-2M^{\sharp}(b\chi_{Q})(y))\chi_{Q}=2[b,M^{\sharp}](\chi_{Q})(y).
\end{align*}
Since $[b,M^{\sharp}]$ is bounded on $(E_{p}^{q})_{t}(\mathbb{R}^{n}),$ then by applying Lemma 2.3, we get
\begin{align*}
\left\|(b-2M^{\sharp}(b\chi_{Q}))\chi_{Q}\right\|^{q}_{(E_{p}^{q})_{t}(\mathbb{R}^{n})}&= 2^{q}\left\|[b,M^{\sharp}](\chi_{Q})\right\|^{q}_{(E_{p}^{q})_{t}(\mathbb{R}^{n})}\\
&\leq C\left\|\chi_{Q}\right\|^{q}_{(E_{p}^{q})_{t}(\mathbb{R}^{n})}\\
&\leq C|Q|.
\end{align*}
 Then we obtain (T.3).

$(T.3)\Longrightarrow (T.4).$ For any fixed cube $Q\subset \mathbb{R}^{n},$~using Lemmas 2.1, 2.3 and statement (T.3) again, we have
\begin{align*}
\frac{1}{|Q|}\int_{Q}|b(y)-2M^{\sharp}(b\chi_{Q})(y)|dy&\leq \left(\frac{1}{|Q|}\int_{Q}|b(y)-2M^{\sharp}(b\chi_{Q})(y)|^{p}dy\right)^{1/p}\left(\frac{1}{|Q|}\int_{Q}|\chi_{Q}|^{p'}dy\right)^{1/p'}\\
&\leq \frac{C}{|Q|}\|b-2M^{\sharp}(b\chi_{Q})\|_{(E_{p}^{q})_{t}(\mathbb{R}^{n})}\|\chi_{Q}\|_{(E_{p'}^{q'})_{t}(\mathbb{R}^{n})}\\
&\leq  C.
\end{align*}
which implies statement (T.4)  since the constant $C$ is independent of $Q.$
Therefore, we finish the proof of Theorem 1.3.
\end{refproof}

\section*{Funding information}

\noindent This work was partly supported by the Fundamental Research Funds for Education Department of Heilongjiang Province (No. 1454YB020), the Fundamental Research Funds for Education Department of Heilongjiang Province (No.2019-KYYWF-0909, 1453ZD031, SJGY20220609) the Reform and Development Foundation for Local Colleges and Universities of the Central Government (No. 2020YQ07).
\section*{Conflict of interest}

\noindent The authors state that there is no conflict of interest.
\section*{Date availability statement}

\noindent All data generated or analysed during this study are included in this published article.
\section*{Author contributions}

\noindent All authors contributed equally to the writing of this article. All authors read the final manuscript and approved its submission.

\bigskip

\end{document}